\documentclass[12pt,a4paper]{article}
\usepackage{indentfirst}
\usepackage{amsfonts}
\usepackage{amssymb}
\usepackage{mathrsfs}
\usepackage{amsmath}
\usepackage{amsthm}
\usepackage{enumerate}
\usepackage{cite}
\usepackage{cases}
\allowdisplaybreaks
\newtheorem{defn}{Definition}[section]
\newtheorem{thm}[defn]{Theorem}
\newtheorem{lem}[defn]{Lemma}

\newtheorem{cor}[defn]{Corollary}
\newtheorem{ex}[defn]{Example}
\newtheorem{re}[defn]{Remark}
\begin{document}
\title{{\bf Biderivations and triple homomorphisms on perfect Jordan algebras}}
\author{\normalsize \bf Chenrui Yao,  Yao Ma,  Liangyun Chen}
\date{{\small{ School of Mathematics and Statistics,  Northeast Normal University,\\ Changchun 130024, CHINA
}}} \maketitle
\date{}

   {\bf\begin{center}{Abstract}\end{center}}

In this paper, we mainly study a class of biderivations and triple homomorphisms on perfect Jordan algebras. Let $J$ be a Jordan algebra and $\delta :J \times J \rightarrow J$ a symmetric biderivation satisfying $\delta(w , u \circ v) = w \cdot \delta(u , v), \forall u,v,w \in J$. If $J$ is perfect and satisfies $Z(J) = \{0\}$, then $\delta$ is of the form $\delta(x , y) = \gamma(x \circ y)$ for all $x , y \in J$, where $\gamma \in Cent(J)$ satisfying $z \cdot \gamma(x \circ y) = x \cdot \gamma(y \circ z) + y \cdot \gamma(x \circ z), \forall x , y , z \in J$. This is the special case of our main theorem which concerns biderivations having their range in a $J$-module. What's more, we give an algorithm which can be applied to find biderivations satisfying $\delta(w , u \circ v) = w \cdot \delta(u , v), \forall u,v,w \in J$ on any Jordan algebra. We also show that for a triple homomorphism between perfect Jordan algebras, $f(x^{2}) = (f(x))^{2}$ or $f(x^{2}) = -(f(x))^{2}$. As an application, such $f$ is a homomorphism if and only if $f(x^{2}) = (f(x))^{2}$. Moreover, we give an algorithm which can be applied to any Jordan algebra.

\renewcommand{\thefootnote}{\fnsymbol{footnote}}
\footnote[0]{ Corresponding author(L. Chen): chenly640@nenu.edu.cn.}
\footnote[0]{Supported by  NNSF of China (Nos. 11771069 and 11801066), NSF of  Jilin province (No. 20170101048JC),   the project of Jilin province department of education (No. JJKH20180005K) and the Fundamental Research Funds for the Central Universities(No. 130014801).}

\section{Introduction}

The concept of Jordan algebras was first introduced by P. Jordan in his study of quantum mechanics. Recall that a Jordan algebra $J$ is an algebra over a field $\rm{F}$ satisfying
\begin{enumerate} [(i)]
\item $x \circ y = y \circ x,\quad\forall x , y \in J$,
\item $(x^{2} \circ y) \circ x = x^{2} \circ (x \circ y),\quad x^{2} = x \circ x ,\quad\forall x , y \in J$.
\end{enumerate}

In \cite{A1 , A2}, A. A. Albert renamed them Jordan algebras and developed a successful structure theory for Jordan algebras over any field of characteristic zero. Using similar research methods in Lie algebras, he proved Lie's theorem, Engel's theorem and Cartan's theorem on solvable Jordan algebras, which are famous on solvable Lie algebras.

As we know, an associative algebra $A$ with the multiplication defined by
\[x \circ y = \frac{1}{2}(xy + yx)\]
forms a Jordan algebra, which is denoted by $A^{+}$. This kind of Jordan algebras is called the special Jordan algebras; the others are called exceptional. In \cite{NJ1}, N. Jacobson classified all finite-dimensional simple Jordan algebras over algebraically closed field $\rm{F}$ whose characteristic is not two:
\begin{enumerate} [(A)]
\item The special Jordan algebra $\rm{F}^{n \times n^{+}}$;
\item $H(\rm{F}^{n \times n} , j_{1})$, where $j_{1}(X) = X^{'}$;
\item
$H(\rm{F}^{2n \times 2n} , j_{S})$, where $j_{S} (X) = S^{'}X^{'}S, S = diag(Q , Q , \cdots , Q)$,
$Q = \begin{pmatrix}
0&1\\
1&0
\end{pmatrix}$;
\item $\rm{F}e_{1} + V$, which is isomorphic to a subalgebra of $Cl(V , f)^{+}$;
\item $H(C^{3 \times 3} , \Gamma), \Gamma = I_{3}$, which is a 27-dimensional exceptional Jordan algebra.
\end{enumerate}
What's more, there is a close relationship between vertex operator algebras and simple Jordan algebras. C. H. Lam constructed vertex operator algebras whose Griess algebras are simple Jordan algebras of type A, B and C in his paper \cite{CL1 , CL2}. As for simple Jordan algebras of type D, A. Takahiro gave the construction in the case that the Jordan algebras are of low rank in \cite{TA1}.

Representation theory has been studied extensively as an important tool to study the structure of an algebra. N. Jacobson gave the definition of the representations and modules of a Jordan algebra in \cite{N1}, which are in the following.

A Jordan module is a system consisting of a vector space $V$, a Jordan algebra $J$, and two compositions $x \cdot a$, $a \cdot x$ for $x$ in $V$, $a$ in $J$ which are bilinear and satisfy
\begin{enumerate} [(i)]
\item $a \cdot x = x \cdot a$,
\item $(x \cdot a) \cdot (b \circ c) + (x \cdot b) \cdot (c \circ a) + (x \cdot c) \cdot (a \circ b) = (x \cdot (b \circ c)) \cdot a + (x \cdot (c \circ a)) \cdot b + (x \cdot (a \circ b)) \cdot c$,
\item $(((x \cdot a) \cdot b) \cdot c) + (((x \cdot c) \cdot b) \cdot a) + x \cdot (a \circ c \circ b) = (x \cdot a) \cdot (b \circ c) + (x \cdot b) \cdot (c \circ a) + (x \cdot c) \cdot (a \circ b)$,
\end{enumerate}
where $x \in V, a , b , c \in J$, and $a_{1} \circ a_{2} \circ a_{3}$ stands for $((a_{1} \circ a_{2}) \circ a_{3})$.

A linear map $a \rightarrow S_{a}$ of a Jordan algebra $J$ into the algebra of linear transformations of a vector space $V$ over $\rm{F}$ is called a representation if for all $a , b , c \in J$,
\begin{enumerate} [(i)]
\item $[S_{a}S_{b \circ c}] + [S_{b}S_{c \circ a}] + [S_{c}S_{a \circ b}] = 0$,
\item $S_{a}S_{b}S_{c} + S_{c}S_{b}S_{a} + S_{(a \circ c) \circ b} = S_{a}S_{b \circ c} + S_{b}S_{c \circ a} + S_{c}S_{a \circ b}$,
\end{enumerate}
where $[AB]$ denotes $AB - BA$.

There is also a lot of research on representations of Jordan algebras. For instance, A. A. Albert obtained a complete result about the special representations of finite-dimensional semi-simple Jordan algebras over a field of characteristic zero in \cite{A1}. What's more, in \cite{PJE1}, P. Jordan found that simple Jordan algebras have only a finite number of non-equivalent irreducible representations which are homomorphisms into a space of linear operators with operation $X \cdot Y = XY + YX$. A. A. Albert proved that the exceptional Jordan algebra $E_{3}$ had no such representations at all in \cite{A3}. F. Marcelo and L. Alicia gave a characterization of representations and irreducible modules on generalized almost-Jordan algebras in \cite{FL1}.

Recall that a derivation on an algebra $A$ is a linear map $D : A \rightarrow A$ satisfying
\[D(xy) = D(x)y + xD(y),\quad\forall x , y \in A.\]

In \cite{ALM1}, A. Carotenuto, L. Dabrowski and M. Dubois-Violette gave the definition of a derivation from a Jordan algebra $J$ to a $J$-module $M$, i.e., the linear map $d : J \rightarrow M$ satisfying
\[d(x \circ y) = x \cdot d(y) + y \cdot d(x),\quad\forall x , y \in J.\]
If we take $M = J$ and $a \cdot x = a \circ x$, then $d$ is the usual derivation of Jordan algebras.

It is well-known that derivations play an indispensable role in the study of the structure of an algebra. And there has been a lot of research on them. B. Jeffrey, I. N. Herstein and L. Charles began the study of contacting the structure of a ring with the special behavior of one of its derivations in \cite{BHL1}. In \cite{BC1}, B. Jeffrey and L. Carini began to study associative rings admitting a derivation with invertible values on some non-central Lie ideals.

As a generalization of derivations, there  are more and more studies on biderivations in recent years. The study of commuting maps and (skew-symmetric) biderivations plays an important role in associative ring theory \cite{M1, MWC}. In E. Daniel's study \cite{DE1}, he described the form of biderivations of the algebra $A \otimes S$ where $A$ denoted a noncommutative unital prime algebra and $S$ denoted a commutative unital algebra over a field $\rm{F}$. In \cite{MK}, M. Ber$\check{s}$ar and K. M. Zhao studied biderivations on Lie algebras by using a more convenient and universal method.

Inspired by \cite{MK}, in this paper, we mainly study symmetric biderivations on Jordan algebras. A biderivation from a Jordan algebra $J$ to a $J$-module $M$ is defined as a bilinear map $\delta : J \times J \rightarrow M$ satisfying
\begin{enumerate} [(i)]
\item $\delta(x \circ y , z) = x \cdot \delta(y , z) + y \cdot \delta(x , z)$,
\item $\delta(x , y \circ z) = y \cdot \delta(x , z) + z \cdot \delta(x , y)$.
\end{enumerate}

If $\delta$ satisfies
\[\delta(x , y) = \delta(y , x),\quad\forall x , y \in J,\]
then $\delta$ is called a symmetric biderivation.

If $\delta$ satisfies
\[\delta(x , y) = -\delta(y , x),\quad\forall x , y \in J,\]
then $\delta$ is called a skew-symmetric biderivation.

It's obvious that $x \mapsto \delta(x , z)$ is a derivation for every $ z \in J$ and $x \mapsto \delta(z , x)$ is also a derivation if $\delta$ is either symmetric or skew-symmetric.

\begin{ex}\label{ex:1.1}
Let $A$ be a commutative associative algebra and $A^{+}$ the special Jordan algebra generated by $A$. Suppose that $D : A \rightarrow A$ is a derivation on $A$. Define $\delta : A \times A \rightarrow A$ by $\delta(x, y) = D(x)D(y)$. One can verify that $\delta$ is a symmetric biderivation on $A$. According to Lemma \ref{le:2.2}, $\delta$ is also a symmetric biderivation on $A^{+}$. Hence, for special Jordan algebras generated by commutative associative algebras, there exist many symmetric biderivations naturally.
\end{ex}

In the following, we need the notion ``centroid" to characterize biderivations. The centroid of a non-associative algebra $A$ is the set
\[\{f \in End(A) | f(ab) = f(a)b = af(b),\quad\forall a , b \in A\},\]
and we denote it by $Cent(A)$.

Centroid can be used to study the structure of algebras. For instance, R. E. Block characterized the structure of differential simple algebra via studying its centroid in \cite{REB1}. V. G. Kac \cite{VGK1} and S. J. Cheng \cite{SJC1} generalized the above result to $Z_{2}$-graded algebras. Naturally, using the centroid, they studied the structure of Lie superalgebras and concluded the theorem of the classification of simple Lie superalgebras, which is viewed as the most classical result in this field.

As a generalization, we now give the definition of the centroid of a $J$-module $M$ for a Jordan algebra $J$, which is denoted by $Cent(M)$. A linear map $\gamma : J \rightarrow M$ belongs to $Cent(M)$ if for all $x , y \in J$,
\[\gamma(x \circ y) = x \cdot \gamma(y).\]



Homomorphisms have been a historic subject for many years. Since homomorphisms can be used to describe the relationship between algebras, it also be widely studied. A homomorphism from $A_{1}$ to $A_{2}$ is a linear map $f : A_{1} \rightarrow A_{2}$ satisfying
\[f(xy) = f(x)f(y),\quad\forall x, y \in A_{1}.\]

As a generalization of homomorphisms, triple homomorphism from a Lie algebra $L_{1}$ to a Lie algebra $L_{2}$ is a linear map $f : L_{1} \rightarrow L_{2}$ satisfying
\[f([[x, y], z]) = [[f(x), f(y)], f(z)],\quad\forall x, y , z \in L_{1}.\]
Similarly, one can define triple homomorphisms between Jordan algebras. Even more interesting is the relationship between homomorphisms and triple homomorphisms. Obviously, both homomorphisms and anti-homomorphisms between Lie algebras are triple homomorphisms. Whether triple homomorphisms are  homomorphisms or anti-homomorphi-sms or not? The answer is given by J. H. Zhou. In \cite{Z1}, he proved that if $L$ is a perfect Lie algebra, then triple homomorphisms are either homomorphisms or anti-homomorphisms or direct sum of homomorphisms and anti-homomorphisms. What is the answer to Jordan algebras?
In this paper, we mainly study on perfect Jordan algebras and get a conclusion that for any triple homomorphism $f$ between prefect Jordan algebras, $f(x^{2}) = (f(x))^{2}$ or $f(x^{2}) = -(f(x))^{2}$(see Theorem \ref{thm:3.3}).

The paper is organized as follows. In Section \ref{se:1}, we first give Lemma \ref{le:2.2} as a proposition about biderivations. Next we reach the main result, Theorem \ref{thm:2.4}, and two example are provided to show that the condition (\ref{eq:1}) in Theorem \ref{thm:2.4} makes sense and isn't constant. Then we prove two important lemmas, Lemma \ref{le:2.10} and Lemma \ref{le:2.12}, which are essential for the algorithm which can be used to find biderivations satisfying (\ref{eq:1}) on any Jordan algebra. At last, we give an example to apply the algorithm.

In Section \ref{se:2}, we'll give a Lemma, Lemma \ref{le:3.1}, which can be used in the following. Next we'll prove our main result Theorem \ref{thm:3.3}. As an application, we come to a conclusion that for a triple homomorphism $f$ between perfect Jordan algebras, $f$ is a homomorphism if and only if $f(x^{2}) = (f(x))^{2}$, written as Corollary \ref{cor:3.4}. In the last, we'll give a theorem, Theorem \ref{thm:3.7}, which can be applied to any Jordan algebra.

Now we give a basic definition.
\begin{defn}
Let $J$ be a Jordan algebra over field $\rm{F}$ and $M$ a $J$-module. For any subset $S$ of $J$, we set
\[Z_{M}(S) = \{v \in M | S \cdot v = 0\}.\]
\end{defn}


We specify below: the symbol ``$\circ$ " stands for the multiplication in Jordan algebras, and the symbol ``$\cdot$ " stands for the module action. We denote the derived algebra of Jordan algebra $J$ by $J^{'}$ and $Z_{J}(J)$ by $Z(J)$.

\section{Biderivations on Jordan algebras}\label{se:1}

Throughout this part, we suppose $char \rm{F} \neq 2 , 3$.

\begin{defn} \cite{M2}
Let $A$ be an associative algebra. A bilinear map $\delta : A \times A \rightarrow A$ is called a biderivation if for all $x , y , z \in A$,
\begin{enumerate} [(i)]
\item $\delta(xy , z) = \delta(x , z)y + x\delta(y , z)$,
\item $\delta(x , yz) = \delta(x , y)z + y\delta(x , z)$.
\end{enumerate}
\end{defn}
\begin{lem}\label{le:2.2}
Let $A$ be an associative algebra and $A^{+}$ the corresponding special Jordan algebra with
\[x \circ y = \frac{1}{2}(xy + yx),\quad\forall x , y \in A.\]
Suppose that $\delta : A \times A \rightarrow A$ is a biderivation on $A$. Then $\delta$ is also a biderivation on $A^{+}$.
\end{lem}
\begin{proof}
For all $x , y , z \in A^{+}$,
\begin{align*}
&\delta(x \circ y , z) = \delta(\frac{1}{2}xy + \frac{1}{2}yx , z)\\
&= \frac{1}{2}\delta(xy , z) + \frac{1}{2}\delta(yx , z)\\
&= \frac{1}{2}\left(\delta(x , z)y + x\delta(y , z)\right) + \frac{1}{2}\left(\delta(y , z)x + y\delta(x , z)\right)\\
&= \frac{1}{2}\left(\delta(x , z)y + y\delta(x , z)\right) + \frac{1}{2}\left(x\delta(y , z) + \delta(y , z)x\right)\\
&= \delta(x , z) \circ y + x \circ \delta(y , z).
\end{align*}
Similarly, we have
\[\delta(x , y \circ z) = \delta(x , y) \circ z + y \circ \delta(x , z).\]
Hence, $\delta$ is also a biderivation on $A^{+}$.
\end{proof}
\begin{cor}\label{cor:2.3}
Let $A$ be a commutative associative algebra and $\delta$ a symmetric biderivation on $A$ satisfying
\[\delta(xy , z) = z\delta(x , y), \quad\forall x , y , z \in A.\]
Then
\[\delta(x \circ y , z) = z \circ \delta(x , y), \quad\forall x , y , z \in A^{+}.\]
\end{cor}
\begin{proof}
For all $x , y , z \in A^{+}$,
\begin{align*}
&\delta(x \circ y , z) = \delta(\frac{1}{2}xy + \frac{1}{2}yx , z) = \frac{1}{2}\delta(xy , z) + \frac{1}{2}\delta(yx , z)\\
&= \frac{1}{2}z\delta(x , y) + \frac{1}{2}z\delta(y , x) = \frac{1}{2}z\delta(x , y) + \frac{1}{2}\delta(x , y)z = z \circ \delta(x , y).
\end{align*}
\end{proof}
\begin{thm}\label{thm:2.4}
Let $J$ be a perfect Jordan algebra over $\rm{F}$ and $M$ a $J$-module satisfying $Z_{M}(J) = \{0\}$. Suppose that $\delta : J \times J \rightarrow M$ is a symmetric biderivation and $\gamma $ belongs to $Cent(M)$. Then the following equations are equivalent.
\begin{equation}\label{eq:1}
\delta(w , u \circ v) = w \cdot \delta(u , v), \quad\forall u,v,w \in J.\tag{1}
\end{equation}
\begin{equation}\label{eq:2}
z \cdot \gamma(x \circ y) = x \cdot \gamma(y \circ z) + y \cdot \gamma(x \circ z), \quad\forall x , y , z \in J.\tag{2}
\end{equation}
\end{thm}
\begin{proof}
First, suppose that $\delta$ is a symmetric biderivation satisfying (\ref{eq:1}). Define $\gamma : J = J^{'} \rightarrow M$ to be a linear map by
\[\gamma(x \circ y) = \delta(x , y),\quad \forall x , y \in J.\]
Suppose $\sum_{i}x_{i} \circ y_{i} = 0$. Then we have
\[0 = \delta\left(u , \sum_{i}x_{i} \circ y_{i}\right) = \sum_{i}\delta(u , x_{i} \circ y_{i}) = \sum_{i}u \cdot \delta(x_{i} , y_{i}) = u \cdot \left(\sum_{i}\delta(x_{i} , y_{i})\right).\]
Since $Z_{M}(J) = \{0\}$, $\sum_{i}\delta(x_{i} , y_{i}) = 0$. Hence, $\gamma$ is well-defined.

For all $u , v \in J$, suppose $v = \sum_{i}x_{i} \circ y_{i}$. Then we have
\begin{align*}
&\delta(u , v) = \delta\left(u , \sum_{i}x_{i} \circ y_{i}\right) = \sum_{i}\delta(u , x_{i} \circ y_{i}) = \sum_{i}u \cdot \delta(x_{i} , y_{i}) = \sum_{i}u \cdot \gamma(x_{i} \circ y_{i})\\
&= u \cdot \left(\sum_{i}\gamma(x_{i} \circ y_{i})\right) = u \cdot \gamma\left(\sum_{i}x_{i} \circ y_{i}\right) = u \cdot \gamma(v).
\end{align*}
Then we have $\gamma(x \circ y) = \delta(x , y) = x \cdot \gamma(y)$, which implies that $\gamma \in Cent(M)$.

For all $x , y , z \in J$,
\begin{align*}
&z \cdot \gamma(x \circ y) = z \cdot \delta(x , y) = \delta(x \circ y , z) = x \cdot \delta(y , z) +  y \cdot \delta(x , z) = x \cdot \gamma(y \circ z) + y \cdot \gamma(x \circ z),
\end{align*}
which implies that $\gamma$ satisfies (\ref{eq:2}).

Now suppose that $\gamma$ belongs to $Cent(M)$ satisfying (\ref{eq:2}). Define $\delta : J \times J \rightarrow M$ to be a bilinear map by
\[\delta(x , y) = \gamma(x \circ y),\quad\forall x , y \in J.\]
For all $x , y , z \in J$,
\begin{align*}
&\delta(x \circ y , z) = \gamma((x \circ y) \circ z) = \gamma(z \circ (x \circ y)) = z \cdot \gamma(x \circ y)\\
&= x \cdot \gamma(y \circ z) + y \cdot \gamma(x \circ z) = x \cdot \delta(y , z) + y \cdot \delta(x , z).
\end{align*}
Hence, $\delta$ is a symmetric biderivation satisfying (\ref{eq:1}).
\end{proof}
\begin{re}
The condition (\ref{eq:1}) is not constant which can be seen from the following examples.
\end{re}
\begin{ex}
Let $A$ be a associative algebra with basis $\{e_{1}, e_{2}, e_{3}\}$ and $A^{+}$ the special Jordan algebra generated by $A$. And the multiplication table is
\[e_{1}e_{1} = e_{1},\; e_{2}e_{2} = e_{2},\; e_{3}e_{2} = e_{2}e_{3} = e_{3},\; e_{i}e_{j} = 0,(others).\]
Obviously, $A$ is commutative and perfect. So is $A^{+}$. Moreover, $Z(A^{+}) = \{0\}$. Suppose $D$ is a derivation on $A$. According to Example \ref{ex:1.1}, $\delta$ is a symmetric biderivation on $A^{+}$ defined by $\delta(x, y) = D(x)D(y)$. From Corollary \ref{cor:2.3}, we have that $\delta$ satisfies (\ref{eq:1}) when $\delta$ satisfies \[\delta(xy , z) = z\delta(x , y), \quad\forall x , y , z \in A.\]
That is to say $\delta$ satisfies (\ref{eq:1}) when $D$ satisfies
\[D(xy)D(z) = zD(x)D(y),\quad\forall x, y, z \in A.\]
One can verify that all derivations on $A$ satisfy the above equation. Hence, $\delta$ is a symmetric biderivation satisfying (\ref{eq:1}) defined by $\delta(x, y) = D(x)D(y)$ where $D$ is a derivation on $A$.
\end{ex}
\begin{ex}
Suppose that $J$ is a Jordan algebra with unit $1$ and satisfies $Z(J) = \{0\}$. It's obvious that $J$ is perfect. For any non trivial symmetric biderivation $\delta$ on $J$, we have
\[\delta(1 \circ 1, u) = 1 \circ \delta(1, u) + 1 \circ \delta(1, u),\quad\forall u \in J,\]
which implies that
\[\delta(1, u) = 0,\quad\forall u \in J.\]
If $\delta$ satisfies (\ref{eq:1}), then we have
\[\delta(u \circ v, 1) = 1 \circ \delta(u, v)\quad\forall u, v \in J,\]
which implies that
\[\delta(u, v) = 0,\quad\forall u, v \in J.\]
That is to say $\delta$ is the zero biderivation. Contradiction.
\end{ex}

We now record a corollary to Theorem \ref{thm:2.4}.
\begin{cor}\label{cor:2.8}
If $J$ is perfect and satisfies $Z(J) = \{0\}$, then every symmetric biderivation $\delta$ on $J$ satisfying (\ref{eq:1}) is of the form $\delta(x , y) = \gamma(x \circ y)$, where $\gamma \in  Cent(J)$ such that (\ref{eq:2}) holds.
\end{cor}
\begin{re}
Let $J$ be a Jordan algebra. It's obvious that $\delta : J \times J \rightarrow Z(J)$ with $\delta(J , J^{'}) = 0$ is a symmetric biderivation on $J$. And we name it trivial biderivation.
\end{re}
Suppose that $\delta : J \times J \rightarrow J$ is an arbitrary symmetric biderivation, for any $x , y \in J, z \in Z(J)$, we have
\[0 = \delta(z \circ x , y) = z \circ \delta(x , y) + x \circ \delta(z , y) = x \circ \delta(z , y),\]
hence $\delta(Z(J) , J) \subseteq Z(J)$.

As a result, setting $\bar{J} = J / Z(J)$, we can define a symmetric biderivation $\bar{\delta} : \bar{J} \times \bar{J} \rightarrow \bar{J}$ by
\[\bar{\delta}(\bar{x} , \bar{y}) = \overline{\delta(x , y)},\]
where $\bar{x} = x + Z(J) \in \bar{J}$ for $x \in J$.
\begin{lem}\label{le:2.10}
Let $J$ be a Jordan algebra. Up to trivial biderivations on $J$, the map $\delta \rightarrow \bar{\delta}$ is a $1$-$1$ map from symmetric biderivations satisfying (\ref{eq:1}) on $J$ to symmetric biderivations satisfying (\ref{eq:1}) on $\bar{J}$.
\end{lem}
\begin{proof}
Let $\delta_{1} , \delta_{2}$ be symmetric biderivations on $J$ such that $\bar{\delta_{1}} = \bar{\delta_{2}}$. Then $\delta = \delta_{1} - \delta_{2}$ is a symmetric biderivation on $J$. Since $\bar{\delta_{1}} = \bar{\delta_{2}}$, we have
\[\bar{\delta_{1}}(\bar{J}) = \bar{\delta_{2}}(\bar{J}).\]
Hence
\[\overline{\delta_{1}(J) - \delta_{2}(J)} = \overline{\delta_{1}(J)} - \overline{\delta_{2}(J)} = \bar{\delta_{1}}(\bar{J}) - \bar{\delta_{2}}(\bar{J}) = 0,\]
which implies
\[\delta_{1}(J) - \delta_{2}(J) \in Z(J),\]
i.e., $\delta (J , J) \subseteq Z(J)$.

Moreover,we have
\[\delta(J , J^{'}) = \delta(J , J \circ J) = J \circ \delta(J , J) = 0.\]
Thus, $\delta$ is a trivial biderivation on $J$.
\end{proof}
\begin{defn}
Let $J$ be a Jordan algebra. A special biderivation is a symmetric biderivation $\delta : J \times J \rightarrow J$ such that
\begin{enumerate}[(i)]
\item $\delta(J^{'} , J^{'}) = 0$,
\item $\delta(J , J^{'}) \subseteq Z_{J}(J^{'})$.
\end{enumerate}
\end{defn}
Every symmetric biderivation $\delta : J \times J \rightarrow J$ satisfies
\begin{equation}\label{eq:3}
\delta(u , x \circ y) = x \circ \delta(u , y) + y \circ \delta(u , x) \in J^{'},\quad \forall u , x , y \in J . \tag{3}
\end{equation}
Thus we have a symmetric biderivation $\delta^{'} : J^{'} \times J^{'} \rightarrow J^{'}$ by restricting $\delta$ to $J^{'} \times J^{'}$.
\begin{lem}\label{le:2.12}
Let $J$ be a Jordan algebra and satisfies $Z(J) = \{0\}$.
\begin{enumerate}[(i)]
\item Up to a special biderivation, every symmetric biderivation $\delta$ satisfying (\ref{eq:1})  on $J$ is an extension of a unique symmetric biderivation  satisfying (\ref{eq:1}) on $J^{'}$.
\item If $J^{'}$ is further perfect, then $J$ has no nonzero special biderivations.
\end{enumerate}
\end{lem}
\begin{proof}
(i) Let $\delta_{1} , \delta_{2}$ be biderivations on $J$ satisfying $\delta_{1}^{'} = \delta_{2}^{'}$. Set $\delta = \delta_{1} - \delta_{2}$. Then
\[\delta(J^{'} , J^{'}) = (\delta_{1} - \delta_{2})(J^{'} , J^{'}) = (\delta_{1}^{'} - \delta_{2}^{'})(J^{'} , J^{'}) = 0.\]
Take $u , y \in J^{'}$ in (\ref{eq:3}) , then we have
\[y \circ \delta(u , x) = 0, \quad \forall x \in J , y , u \in J^{'},\]
i.e., $\delta(J , J^{'}) \subseteq Z_{J}(J^{'})$.

Hence, $\delta$ is a special biderivation on $J$.

(ii) Let $\delta$ be a special biderivation on $J$. From the proof of (i), we have
\[\delta(J , J^{'}) \subseteq Z_{J}(J^{'}).\]
Since $J$ is perfect and satisfies $Z(J) = \{0\}$, we have $\delta(J , J^{'}) = 0$.

Assuming that $\delta \neq 0$, there are $x_{1} , x_{2} \in J$ such that
\[\delta(x_{1} , x_{2}) = z_{12} \neq 0.\]
Since $Z(J) = \{0\}$, we can find $x_{3} \in J$ such that
\[x_{3} \circ z_{12} = z \neq 0.\]
Let
\[\delta(x_{1} , x_{3}) = z_{13} , \delta(x_{2} , x_{3}) = z_{23},\]
then we have
\[0 = \delta(x_{1} \circ x_{3} , x_{2}) = x_{1} \circ \delta(x_{3} , x_{2}) + x_{3} \circ \delta(x_{1} , x_{2}) = x_{1} \circ z_{23} + z,\]
\[0 = \delta(x_{1} \circ x_{2} , x_{3}) = x_{1} \circ \delta(x_{2} , x_{3}) + x_{2} \circ \delta(x_{1} , x_{3}) = x_{1} \circ z_{23} + x_{2} \circ z_{13},\]
\[0 = \delta(x_{2} \circ x_{3} , x_{1}) = x_{2} \circ \delta(x_{3} , x_{1}) + x_{3} \circ \delta(x_{2} , x_{1}) = x_{2} \circ z_{13} + z.\]
We deduce that
\[z = - x_{1} \circ z_{23} = x_{2} \circ z_{13} = -z,\]
which is a contradiction. Hence, $\delta = 0$.
\end{proof}

Let us explain an algorithm for finding all symmetric biderivations satisfying (\ref{eq:1}) on any Jordan algebra $J$ using Lemma \ref{le:2.10} and \ref{le:2.12}. Note that we have a sequence of quotient Jordan algebras:
\begin{equation}\label{eq:4}
J_{(1)} = J , J_{(2)} = J_{(1)} / Z(J_{(1)}) , \cdots , J_{(r+1)} = J_{(r)} / Z(J_{(r)}).\tag{4}
\end{equation}
If there is $r \in N$ such that $Z(J_{(r)}) = 0$, then repeatedly applying Lemma \ref{le:2.10} to the above sequence backward, we reduce the problem of finding symmetric biderivations on $J$ to the problem of finding symmetric biderivations on the Jordan algebras $J_{(r+1)}$. If $J_{(r+1)}$ is also perfect, then using Corollary \ref{cor:2.8} we have all biderivations on $J_{(r+1)}$. If $J_{(r+1)}$ is not perfect, using Lemma \ref{le:2.12} we reduce the problem of finding symmetric biderivations on $J_{(r+1)}$ to the problem of finding symmetric biderivations on the Jordan algebra $J^{'}_{r+1}$. Now we repeat the procedure based on (\ref{eq:4}) with $J$ replaced by $J^{'}_{r+1}$, and continue this algorithm.

Now we give an example to apply the above algorithm.
\begin{ex}
Let $J$ be a Jordan algebra with basis $\{x_{1}, x_{2}, x_{3}\}$. And the multiplication table is
\[x_{1} \circ x_{i} = x_{i} \circ x_{1} = 0,\;1 \leq i \leq 3;\quad x_{i} \circ x_{j} = x_{j} \circ x_{i} = \sum^{3}_{k = 1}x_{k},\;2 \leq i , j \leq 3 .\]
Obviously, $Z(J) = \{x_{1}\}$.
Then according to (\ref{eq:4}), we have $J_{1} = J/Z(J) = \{\overline{x_{2}} , \overline{x_{3}}\}$. Obviously, $J_{1}$ is perfect and $Z(J_{1}) = 0$, using Corollary \ref{cor:2.8} we have all biderivations satisfying (\ref{eq:1}) on $J_{1}$ , i.e., they are in the form of $\delta(x , y) = \gamma(x \circ y)$, where $\gamma \in Cent (J_{1})$ and satisfies (\ref{eq:2}).
\end{ex}

\section{Triple Homomorphisms on Jordan algebras}\label{se:2}

Suppose that $f$ is a triple homomorphism from $J_{1}$ to $J_{2}$ where $J_{1}$ and $J_{2}$ are Jordan algebras. Denote $Ann_{f}(J_{2})$ the set
\[\{a \in J_{2} | a \circ f(x) = 0,\quad\forall x \in J_{1}\}.\]
\begin{lem}\label{le:3.1}
Suppose $J_{1}$ is a perfect Jordan algebra over $\rm{F}$and $f$ is a triple homomorphism from $J_{1}$ to $J_{2}$ satisfying $Ann_{f}(J_{2}) = \{0\}$ where $J_{2}$ is an arbitrary Jordan algebra over $\rm{F}$. There exists an $\rm{F}$-linear mapping $\delta_{f} : J_{1} \rightarrow J_{2}$ such that for all $x \in J_{1}$ with $x = \sum_{i \in I}(x_{1i} \circ x_{2i})(x_{1i}, x_{2i} \in J_{1})$, $\delta_{f}(x) = \sum_{i \in I}(f(x_{1i}) \circ f(x_{2i}))$.
\end{lem}
\begin{proof}
It's sufficient to prove that $\sum_{i \in I}(f(x_{1i}) \circ f(x_{2i}))$ is independent of the expression of $x$.

Suppose that
\[x = \sum_{i \in I}(x_{1i} \circ x_{2i}) = \sum_{j \in H}(x_{1j} \circ x_{2j}).\]
Let
\[\alpha = \sum_{i \in I}(f(x_{1i}) \circ f(x_{2i})), \beta = \sum_{j \in H}(f(x_{1j}) \circ f(x_{2j})).\]
For any $z \in J_{1}$, we have
\begin{align*}
&f(z) \circ (\alpha - \beta) = f(z) \circ \left(\sum_{i \in I}(f(x_{1i}) \circ f(x_{2i})) - \sum_{j \in H}(f(x_{1j}) \circ f(x_{2j}))\right)\\
&= f(z) \circ \left(\sum_{i \in I}(f(x_{1i}) \circ f(x_{2i}))\right) - f(z) \circ \left(\sum_{j \in H}(f(x_{1j}) \circ f(x_{2j}))\right)\\
&= \sum_{i \in I}(f(z) \circ (f(x_{1i}) \circ f(x_{2i}))) - \sum_{j \in H}(f(z) \circ (f(x_{1j}) \circ f(x_{2j})))\\
&= \sum_{i \in I}f(z \circ (x_{1i} \circ x_{2i})) - \sum_{j \in H}f(z \circ (x_{1j} \circ x_{2j}))\\
&= f\left(z \circ \left(\sum_{i \in I}(x_{1i} \circ x_{2i})\right)\right) - f\left(z \circ \left(\sum_{j \in H}(x_{1j} \circ x_{2j})\right)\right)\\
&= f(z \circ x) - f(z \circ x) = 0,
\end{align*}
we have $\alpha - \beta = 0$, i.e., $\alpha = \beta$. This completes the proof.
\end{proof}
\begin{ex}\label{ex:3.2}
Let $J$ be the Jordan algebra with basis $\{1, u_{1}, \cdots, u_{n}\}$ where $1$ is the unit, $u_{i}^{2} = \alpha_{i}1(\alpha_{i} \neq 0)(1 \leq i \leq n)$ and $u_{i} \circ u_{j} = 0(i \neq j;\;1 \leq i, j \leq n)$. Define $f : J \rightarrow J$ to be a linear map by
\begin{align*}
&f(1) = -1,\\
&f(u_{i}) = u_{i},(1 \leq i \leq n)
\end{align*}

One can verify that $f$ is a triple homomorphism from $J$ to $J$. Moreover, $f$ satisfies that $Ann_{f}(J) = \{0\}$.
\end{ex}
\begin{thm}\label{thm:3.3}
Suppose $J_{1}$ is a perfect Jordan algebra over $\rm{F}$and $f$ is a triple homomorphism from $J_{1}$ to $J_{2}$ satisfying $Ann_{f}(J_{2}) = \{0\}$ where $J_{2}$ is an arbitrary Jordan algebra over $\rm{F}$. Then $f(x^{2}) = (f(x))^{2}$ or $f(x^{2}) = -(f(x))^{2}$, $\forall x \in J_{1}$.
\end{thm}
\begin{proof}
For any $x, y \in J_{1}$, we have
\[f(x^{2} \circ y) \circ f(x) = \delta_{f}((x^{2} \circ y) \circ x) = \delta_{f}(x^{2} \circ (y \circ x)) = f(x^{2}) \circ f(y \circ x).\]
On the other hand,
\[f(x^{2} \circ y) \circ f(x) = ((f(x) \circ f(x)) \circ f(y)) \circ f(x) = (f(x) \circ f(x)) \circ (f(y) \circ f(x)).\]
Hence, we have
\[f(x^{2}) \circ f(y \circ x) = (f(x) \circ f(x)) \circ (f(y) \circ f(x)).\]
In particular, take $x = y$, we have
\[(f(x^{2}))^{2} = (f(x) \circ f(x))^{2},\]
which implies that
\[f(x^{2}) = (f(x))^{2}\; or\; f(x^{2}) = -(f(x))^{2}.\]
\end{proof}
\begin{cor}\label{cor:3.4}
Suppose $J_{1}$ is a perfect Jordan algebra over $\rm{F}$and $f$ is a triple homomorphism from $J_{1}$ to $J_{2}$ satisfying $Ann_{f}(J_{2}) = \{0\}$ where $J_{2}$ is an arbitrary Jordan algebra over $\rm{F}$. Then $f$ is a homomorphism if and only if $f(x^{2}) = (f(x))^{2}$, $\forall x \in J_{1}$ when $char \rm{F} \neq 2$.
\end{cor}
\begin{proof}
Suppose that $f$ is a homomorphism, then for any $x, y \in J_{1}$, we have
\[f(x \circ y) = f(x) \circ f(y).\]
In particular, take $x = y$, we have $f(x^{2}) = f(x) \circ f(x) = (f(x))^{2}$.

Now suppose that $f(x^{2}) = (f(x))^{2}$, $\forall x \in J_{1}$.

Replace $x$ by $x + y$ where $y \in J_{1}$ in the above equation. Then we have
\[f((x + y)^{2}) = f(x + y) \circ f(x + y),\]
\[f(x^{2}) + 2f(x \circ y) + f(y^{2}) = f(x) \circ f(x) + 2f(x) \circ f(y) + f(y) \circ f(y),\]
since $char \rm{F} \neq 2$, we have
\[f(x \circ y) = f(x) \circ f(y),\]
which implies that $f$ is a homomorphism.
\end{proof}
\begin{re}
The triple homomorphism in example \ref{ex:3.2} satisfies $f(x^{2}) = -(f(x))^{2}$ for all $x \in J$.
\end{re}

Next we'll study triple homomorphisms on any Jordan algebra. At first, we'll give a definition about our study.
\begin{defn}
Suppose that $J_{1}$, $J_{2}$ are Jordan algebras over $\rm{F}$ and $f : J_{1} \rightarrow J_{2}$ is a triple homomorphism. Then $f$ is called a special triple homomorphism if $f(J_{1}^{''}) = 0$ where $J_{1}^{''}$ denotes all elements of type $(x \circ y) \circ z, \; x, y, z \in J_{1}$.
\end{defn}

It's obvious that every triple homomorphism $f : J_{1} \rightarrow J_{2}$ satisfies $f(J_{1}^{''}) \subseteq J_{2}^{''}$. Thus we have a triple homomorphism $f^{'} : J_{1}^{''} \rightarrow J_{2}^{''}$by restricting $f$ to $J_{1}^{''}$.
\begin{thm}\label{thm:3.7}
Let $J_{1}$, $J_{2}$ be Jordan algebras over $\rm{F}$.
\begin{enumerate}[(1)]
\item Up to a special triple homomorphism, any triple homomorphism $f$ from $J_{1}$ to $J_{2}$ is an extension of a unique triple homomorphism from $J_{1}^{''}$ to $J_{2}^{''}$.
\item If $J_{1}$ is perfect, then there no nonzero special triple homomorphism from $J_{1}$ to $J_{2}$ where $J_{2}$ is an arbitrary Jordan algebra.
\end{enumerate}
\end{thm}
\begin{proof}
(1). Suppose that $f_{1}$, $f_{2}$ are two triple homomorphisms from $J_{1}$ to $J_{2}$ such that $f_{1}^{'} = f_{2}^{'}$. Let $f = f_{1} - f_{2}$. Then we have
\[f(J_{1}^{''}) = (f_{1} - f_{2})(J_{1}^{''}) = (f_{1}^{'} - f_{2}^{'})(J_{1}^{''}) = 0,\]
which implies that $f$ is a special triple homomorphism.

(2). Suppose that $f$ is a special triple homomorphism from $J_{1}$ to $J_{2}$, i.e., $f(J_{1}^{''}) = 0$. Note that $J_{1}$ is perfect, we have
\[f(J_{1}) = f(J_{1}^{''}) = 0,\]
which implies that $f$ is the zero homomorphism.
\end{proof}

Now we will give an algorithm for finding all triple homomorphisms from any Jordan algebra $J_{1}$ to $J_{2}$ using Theorem \ref{thm:3.7}. If $J_{1}$ is not perfect, using Theorem \ref{thm:3.7} we reduce the problem of finding triple homomorphisms from $J_{1}$ to $J_{2}$ to the problem of finding triple homomorphisms from $J_{1}^{''}$ to $J_{2}^{''}$.

\end{document}